\documentclass[12pt]{amsart}
\usepackage{amsmath,amsthm,amsfonts,amssymb,mathrsfs}
\usepackage{color}

\usepackage{tikz}

\usepackage{amssymb,cite}
\usepackage[colorlinks,plainpages,citecolor=magenta, linkcolor=blue, backref]{hyperref}

\usepackage{hyperref}
\date{\today}

\usepackage{hyperref}
\input{xy}
 \xyoption{all}
 \xyoption{arc}

\newtheorem{theorem}{Theorem}

\newtheorem{proposition}{Proposition}

\newtheorem{lemma}{Lemma}
\theoremstyle{definition}
\newtheorem{definition}{Definition}
\newtheorem{example}{Example}
\newtheorem{remark}{Remark}

\begin{document}

\title[On the semigroup of monoid endomorphisms of the semigroup $\mathscr{C}_{+}(a,b)$]{On the semigroup of monoid endomorphisms of the semigroup $\mathscr{C}_{+}(a,b)$}
\author{Oleg Gutik and Sher-Ali Penza}
\address{Ivan Franko National University of Lviv, Universytetska 1, Lviv, 79000, Ukraine}
\email{oleg.gutik@lnu.edu.ua, SHER-ALI.PENZA@lnu.edu.ua}



\begin{abstract}
Let $\mathscr{C}_{+}(a,b)$ be the submonoid of the bicyclic monoid which is studied in \cite{Makanjuola-Umar=1997}.
We describe monoid endomorphisms of the semigroup $\mathscr{C}_{+}(a,b)$ which are generated by the family of all congruences of the bicyclic monoid and all injective monoid endomorphisms of $\mathscr{C}_{+}(a,b)$.
\end{abstract}

\maketitle


\section*{Introduction}
We shall follow the terminology of~\cite{Clifford-Preston-1961, Clifford-Preston-1967, Lawson=1998}. By $\omega$ we denote the set of all non-negative integers, by $\mathbb{N}$ the set of all positive integers.

A semigroup $S$ is called {\it inverse} if for any
element $x\in S$ there exists a unique $x^{-1}\in S$ such that
$xx^{-1}x=x$ and $x^{-1}xx^{-1}=x^{-1}$. The element $x^{-1}$ is
called the {\it inverse of} $x\in S$. If $S$ is an inverse
semigroup, then the function $\operatorname{inv}\colon S\to S$
which assigns to every element $x$ of $S$ its inverse element
$x^{-1}$ is called the {\it inversion}.

If $S$ is a semigroup, then we shall denote the subset of all
idempotents in $S$ by $E(S)$. If $S$ is an inverse semigroup, then
$E(S)$ is closed under multiplication and we shall refer to $E(S)$ as a
\emph{band} (or the \emph{band of} $S$). Then the semigroup
operation on $S$ determines the following partial order $\preccurlyeq$
on $E(S)$: $e\preccurlyeq f$ if and only if $ef=fe=e$. This order is
called the {\em natural partial order} on $E(S)$. A \emph{semilattice} is a commutative semigroup of idempotents.

If $S$ is an inverse semigroup then the semigroup operation on $S$ determines the following partial order $\preccurlyeq$
on $S$: $s\preccurlyeq t$ if and only if there exists $e\in E(S)$ such that $s=te$. This order is
called the {\em natural partial order} on $S$ \cite{Wagner-1952}.

The bicyclic monoid ${\mathscr{C}}(p,q)$ is the semigroup with the identity $1$ generated by two elements $p$ and $q$ subjected only to the condition $pq=1$. The semigroup operation on ${\mathscr{C}}(p,q)$ is determined as
follows:
\begin{equation*}
    q^kp^l\cdot q^mp^n=q^{k+m-\min\{l,m\}}p^{l+n-\min\{l,m\}}.
\end{equation*}
It is well known that the bicyclic monoid ${\mathscr{C}}(p,q)$ is a bisimple (and hence simple) combinatorial $E$-unitary inverse semigroup and every non-trivial congruence on ${\mathscr{C}}(p,q)$ is a group congruence \cite{Clifford-Preston-1961}.

Let $\mathfrak{h}\colon S\to T$ be a homomorphism of semigroups. Then for any $s\in S$ and $A\subseteq S$ by $(s)\mathfrak{h}$ and $(A)\mathfrak{h}$ we denote the images of $s$ and $A$, respectively, under the homomorphism $\mathfrak{h}$. Also, for any $t\in T$ by $(s)\mathfrak{h}^{-1}$ we denote the full preimage of $s$ under the map $\mathfrak{h}$. A homomorphism  $\mathfrak{h}\colon S\to T$ of monoids which preserves the unit elements of $S$ is called a \emph{monoid homomorphism}.  A homomorphism  $\mathfrak{h}\colon S\to S$ of a semigroup (a monoid) is called an endomorphism (a monoid endomorphism) of $S$, and in the case when $\mathfrak{h}$ is an isomorphism then $\mathfrak{h}$ is said to be an \emph{automorphism} of $S$.

It is well-known that every automorphism of the bicyclic monoid ${\mathscr{C}}(p,q)$  is the identity self-map of ${\mathscr{C}}(p,q)$ \cite{Clifford-Preston-1961}, and hence the group \linebreak $\mathbf{Aut}({\mathscr{C}}(p,q))$ of automorphisms of ${\mathscr{C}}(p,q)$ is trivial. In \cite{Gutik-Prokhorenkova-Sekh=2021} all endomorphisms of the bicyclic semigroup are described and it is proved that the semigroups $\mathrm{\mathbf{End}}({\mathscr{C}}(p,q))$ of all endomorphisms of the bicyclic semigroup ${\mathscr{C}}(p,q)$ is isomorphic to the semidirect products $(\omega,+)\rtimes_\varphi(\omega,*)$, where $+$ and $*$ are the usual addition and the usual multiplication on $\omega$.

Subsemigroups of then bicyclic monoid were studied in \cite{Descalco-Ruskuc-2005, Descalco-Ruskuc-2008, Makanjuola-Umar=1997}.
In \cite{Makanjuola-Umar=1997}  the following anti-isomorphic subsemigroups of the bicyclic monoid
\begin{equation*}
  \mathscr{C}_{+}(a,b)=\left\{b^ia^j\in\mathscr{C}(a,b)\colon i\leqslant j,\, i,j\in\omega\right\}
  \end{equation*}
and
\begin{equation*}
  \mathscr{C}_{-}(a,b)=\left\{b^ia^j\in\mathscr{C}(a,b)\colon i\geqslant j,\, i,j\in\omega\right\}
\end{equation*}
are studied. In the paper \cite{Gutik=2023} topologizations of the semigroups $\mathscr{C}_{+}(a,b)$ and $\mathscr{C}_{-}(a,b)$ are studied.

Later in this paper by $\mathfrak{End}(\mathscr{C}_{+}(a,b))$  we denote the semigroup of all monoid endomorphisms of the semigroup $\mathscr{C}_{+}(a,b)$.

In this paper we describe monoid endomorphisms of the semigroup $\mathscr{C}_{+}(a,b)$ which are generated by the family of all congruences of the bicyclic monoid and all injective monoid endomorphisms of $\mathscr{C}_{+}(a,b)$.

\section{On monoid endomorphisms of $\mathscr{C}_{+}(a,b)$ which are restrictions of homomorphisms of the bicyclic monoid}\label{section-2}

In \cite{Gutik-Prokhorenkova-Sekh=2021} was proved that every monoid endomorphism $\varepsilon\colon \mathscr{C}(a,b)\to \mathscr{C}(a,b)$ of the bicyclic monoid is one of the following forms:
\begin{enumerate}
  \item[$(i)$] $\varepsilon=\lambda_k$ for some positive integer $k$, where $(b^ia^j)\lambda_k=b^{ki}a^{kj}$ for any $i,j\in\omega$;
  \item[$(ii)$] $\varepsilon=\lambda_0$ is the annihilating endomorphism of $\mathscr{C}(a,b)$, i.e.,  $(b^ia^j)\lambda_0=1$ for any $i,j\in\omega$.
\end{enumerate}

Simple verifications show that in the both cases each of these monoid endomorphisms of the bicyclic semigroup induces the monoid endomorphism of $\mathscr{C}_{+}(a,b)$, which we denote by the similar way:
\begin{equation*}
  \lambda_k\colon \mathscr{C}_{+}(a,b)\to \mathscr{C}_{+}(a,b), b^ia^j\mapsto b^{ki}a^{kj}, \; i,j\in\omega,
\end{equation*}
for some $k\in\omega$.

For any $k_1,k_2\in \omega$ we have that
\begin{equation*}
  (b^ia^j)(\lambda_{k_1}\circ\lambda_{k_2})=((b^ia^j)\lambda_{k_1})\lambda_{k_2}=(b^{k_1i}a^{k_1j})\lambda_{k_2}=b^{k_1k_2i}a^{k_1k_2j}, \, i,j\in\omega.
\end{equation*}
This implies that $\lambda_{k_1}\circ\lambda_{k_2}=\lambda_{k_1k_2}$ for all $k_1,k_2\in \omega$, and hence the set $\left\{\lambda_k\colon k\in\omega\right\}$ of endomorphisms of $\mathscr{C}_{+}(a,b)$ is closed under the operation of composition.

By $\mathfrak{End}_{\langle\lambda\rangle}(\mathscr{C}_{+}(a,b))$  we denote the subsemigroup of $\mathfrak{End}(\mathscr{C}_{+}(a,b))$, which is generated by the family $\left\{\lambda_k\colon k\in\omega\right\}$ of endomorphisms of the monoid $\mathscr{C}_{+}(a,b)$.

\begin{proposition}\label{proposition-2.1}
The semigroup $\mathfrak{End}_{\langle\lambda\rangle}(\mathscr{C}_{+}(a,b))$ is isomorphic to the multiplicative semigroup $(\omega,*)$ of non-negative integers.
\end{proposition}

\begin{proof}
We define the map $\mathfrak{I}\colon \mathfrak{End}_{\langle\lambda\rangle}(\mathscr{C}_{+}(a,b)) \to (\omega,*)$  by the formula $(\lambda_k)\mathfrak{I}=k$ for any $\lambda_k\in \mathfrak{End}_{\langle\lambda\rangle}(\mathscr{C}_{+}(a,b))$. The above arguments and simple verifications show that so defined map $\mathfrak{I}$ is a semigroup isomorphism.
\end{proof}

It is well known that any inverse semigroup $S$ admits the \emph{smallest} (\emph{minimal}) \emph{group congruence} $\mathfrak{C}_{\mathrm{mg}}$: $a\mathfrak{C}_{\mathrm{mg}}b$ if and only if there exists $e\in E(S)$ such that $ea=eb$ (see \cite{Lawson=1998}). The smallest group congruence $\mathfrak{C}_{\mathrm{mg}}$ on the bicyclic semigroup $\mathscr{C}(a,b)$ is determined in the following way: $b^{i_1}a^{j_1}\mathfrak{C}_{\mathrm{mg}}b^{i_2}a^{j_2}$ if and only if $i_1-j_1=i_2-j_2$ \cite{Lawson=1998}. Since the quotient semigroup $\mathscr{C}(a,b)/\mathfrak{C}_{\mathrm{mg}}$ is isomorphic to the additive group of integers $\mathbb{Z}(+)$, the natural homomorphism $\mathfrak{h}_{\mathfrak{C}_{\mathrm{mg}}}\colon \mathscr{C}(a,b)\to \mathscr{C}(a,b)/\mathfrak{C}_{\mathrm{mg}}$ generates the homomorphism $\mathfrak{h}_{\mathfrak{C}_{\mathrm{mg}}}\colon \mathscr{C}(a,b)\to \mathbb{Z}(+)$ by the formula $(b^ia^j)\mathfrak{h}_{\mathfrak{C}_{\mathrm{mg}}}=j-i$, $i,j\in\omega$. By $(\omega,+)$ we denote the additive group of non-negative integers. This implies that the restriction $\mathfrak{h}_{\mathfrak{C}_{\mathrm{mg}}}{\upharpoonleft}_{\mathscr{C}_{+}(a,b)}\colon \mathscr{C}_{+}(a,b)\to (\omega,+)$ of the homomorphism $\mathfrak{h}_{\mathfrak{C}_{\mathrm{mg}}}$ is a homomorphism, as well.

\begin{lemma}\label{lemma-2.2}
For any $i,j,k\in\omega$ with $j\geqslant i$ the set
$$
S_{i,j,k}=\left\{b^ia^i\right\}\cup\big\{(b^ja^{j+k})^n\colon n\in\mathbb{N}\big\}
$$
with the induced semigroup operation from the bicyclic monoid $\mathscr{C}(a,b)$ is isomorphic to the semigroup $(\omega,+)$.
\end{lemma}

\begin{proof}
We define the mapping $\mathfrak{J}_{i,j,k}\colon (\omega,+)\to S_{i,j,k}$ by the formula
\begin{equation*}
  (n)\mathfrak{J}_{i,j,k}=
  \left\{
    \begin{array}{ll}
      b^ia^i,         & \hbox{if~} n=0;\\
      (b^ja^{j+k})^n, & \hbox{if~} n>0.
    \end{array}
  \right.
\end{equation*}
Simple verifications show that such defined map $\mathfrak{J}_{i,j,k}$ is a bijective homomorphism.
\end{proof}

\begin{definition}\label{definition-2.3}
For arbitrary $l\in\omega$ and $m\in\mathbb{N}$ we define the map $\sigma_{l,m}\colon\mathscr{C}_{+}(a,b)\to S_{0,l,m}$ by the formula $(b^ia^j)\sigma_{l,m}=((b^ia^j)\mathfrak{h}_{\mathfrak{C}_{\mathrm{mg}}})\mathfrak{J}_{0,l,m}$, $i\leqslant j$, $i,j\in\omega$. Since $\mathfrak{h}_{\mathfrak{C}_{\mathrm{mg}}}$  and $\mathfrak{J}_{0,l,m}\colon (\omega,+)\to S_{0,l,m}$ are homomorphisms, $\sigma_{l,m}$ is a homomorphism, too.
Simple verifications show that
\begin{equation*}
(b^ia^j)\sigma_{l,m}=
\left\{
  \begin{array}{ll}
    1,               & \hbox{if~} i=j; \\
    b^la^{l+m(j-i)}, & \hbox{if~} i<j
  \end{array}
\right.
\end{equation*}
for all $i,j\in\omega$.
\end{definition}

We observe that every elements of the semigroup $\mathscr{C}_{+}(a,b)$ can be represented in the form $b^{i}a^{i+j}$ for some $i,j\in\omega$. Then for any positive integer $n$ we have that
\begin{align*}
  (b^{i}a^{i+j})^n&=b^{i}a^{i+j}\cdot b^{i}a^{i+j}\cdot(b^{i}a^{i+j})^{n-2}= \\
   &=b^{i}a^{i+2j}\cdot(b^ia^j)^{n-2}=\\
   &=\cdots=\\
   &=b^{i}a^{i+nj}.
\end{align*}

\begin{lemma}\label{lemma-2.4}
$\sigma_{l_1,m_1}\circ\sigma_{l_2,m_2}=\sigma_{l_2,m_1m_2}$ for arbitrary $l_1,l_2\in\omega$ and $m_1,m_2\in\mathbb{N}$.
\end{lemma}

\begin{proof}
Fix an arbitrary  $b^ia^{j}\in \mathscr{C}_{+}(a,b)$, $i,j\in\omega$. Then we have that
\begin{align*}
  ((b^ia^j)\sigma_{l_1,m_1})\sigma_{l_2,m_2}&=
\left\{
  \begin{array}{ll}
    (1)\sigma_{l_2,m_2},                          & \hbox{if~} i=j;\\
    ((b^{l_1}a^{l_1+m_1})^{j-i})\sigma_{l_2,m_2}, & \hbox{if~} i<j
  \end{array}
\right.=\\
   &=
   \left\{
  \begin{array}{ll}
    1,                                         & \hbox{if~} i=j;\\
    (b^{l_1}a^{l_1+(j-i)m_1})\sigma_{l_2,m_2}, & \hbox{if~} i<j
  \end{array}
\right.=\\
   &=
   \left\{
  \begin{array}{ll}
    1,                               & \hbox{if~} i=j;\\
    (b^{l_2}a^{l_2+m_2})^{(j-i)m_1}, & \hbox{if~} i<j
  \end{array}
\right.=\\
   &=
   \left\{
  \begin{array}{ll}
    1,                          & \hbox{if~} i=j;\\
    b^{l_2}a^{l_2+(j-i)m_1m_2}, & \hbox{if~} i<j
  \end{array}
\right.=\\
   &=(b^ia^j)\sigma_{l_2,m_1m_2}.
\end{align*}

\end{proof}

By $\mathfrak{End}_{\langle\sigma\rangle}(\mathscr{C}_{+}(a,b))$  we denote the subsemigroup of $\mathfrak{End}(\mathscr{C}_{+}(a,b))$, which is generated by the family $\left\{\sigma_{l,m}\colon l,m\in\omega, \; m>0\right\}$ of endomorphisms of the monoid $\mathscr{C}_{+}(a,b)$.

By $\mathfrak{RZ}(\omega)$ we denote the set $\omega$ with the right-zero multiplication, i.e., $xy=y$ for all $x,y\in\omega$, and by $(\mathbb{N},*)$ the multiplicative semigroup of positive integers. We define the map $\mathfrak{I}\colon\mathfrak{End}_{\langle\sigma\rangle}(\mathscr{C}_{+}(a,b))\to \mathfrak{RZ}(\omega)\times (\mathbb{N},*)$ by the formula $(\sigma_{l,m})\mathfrak{I}=(l,m)$, $l\in\omega$, $m\in\mathbb{N}$.
Lemma~\ref{lemma-2.4} implies that such defined map $\mathfrak{I}$ is a semigroup homomorphism, and moreover $\mathfrak{I}$ is bijective. Hence we get the following proposition.

\begin{proposition}\label{proposition-2.5}
The semigroup $\mathfrak{End}_{\langle\sigma\rangle}(\mathscr{C}_{+}(a,b))$ is isomorphic to the direct product $\mathfrak{RZ}(\omega)\times (\mathbb{N},*)$.
\end{proposition}

Fix an arbitrary  $b^ia^{j}\in \mathscr{C}_{+}(a,b)$, $i,j\in\omega$, $j\geqslant i$. Then for any $\sigma_{l,m}\in \mathfrak{End}_{\langle\sigma\rangle}(\mathscr{C}_{+}(a,b))$ and any $\lambda_k\in \mathfrak{End}_{\langle\lambda\rangle}(\mathscr{C}_{+}(a,b))$ we have that
\begin{align*}
  ((b^ia^j)\sigma_{l,m})\lambda_k &=
\left\{
  \begin{array}{ll}
    (1)\lambda_k,                          & \hbox{if~} i=j;\\
    ((b^{l}a^{l+m})^{j-i})\lambda_k, & \hbox{if~} i<j
  \end{array}
\right.=\\
   &=
\left\{
  \begin{array}{ll}
    1,                               & \hbox{if~} i=j;\\
    (b^{l}a^{l+m(j-i)})\lambda_k, & \hbox{if~} i<j
  \end{array}
\right.=\\
   &=
\left\{
  \begin{array}{ll}
    1,                    & \hbox{if~} i=j;\\
    b^{kl}a^{kl+km(j-i)}, & \hbox{if~} i<j
  \end{array}
\right.=\\
   &=(b^ia^j)\sigma_{kl,km}
\end{align*}
and
\begin{align*}
  ((b^ia^j)\lambda_k)\sigma_{l,m} &=(b^{ki}a^{kj})\sigma_{l,m}= \\
   &=
 \left\{
  \begin{array}{ll}
    (b^{ki}a^{ki})\sigma_{l,m}, & \hbox{if~} ki=kj;\\
    (b^{l}a^{l+m})^{kj-ki},     & \hbox{if~} ki<kj
  \end{array}
\right.=  \\
   &=
 \left\{
  \begin{array}{ll}
    (b^{ki}a^{ki})\sigma_{l,m}, & \hbox{if~} i=j;\\
    (b^{l}a^{l+m})^{kj-ki},     & \hbox{if~} i<j
  \end{array}
\right.=  \\
   &=
\left\{
  \begin{array}{ll}
    1,                    & \hbox{if~} i=j;\\
    b^{l}a^{l+km(j-i)}, & \hbox{if~} i<j
  \end{array}
\right.=\\
   &=(b^ia^j)\sigma_{l,km}.
\end{align*}

This implies the following

\begin{proposition}\label{proposition-2.6}
$\sigma_{l,m}\lambda_k=\sigma_{kl,km}$ and $\lambda_k\sigma_{l,m}=\sigma_{l,km}$
for any $\sigma_{l,m}\in \mathfrak{End}_{\langle\sigma\rangle}(\mathscr{C}_{+}(a,b))$ and any $\lambda_k\in \mathfrak{End}_{\langle\lambda\rangle}(\mathscr{C}_{+}(a,b))\setminus\{\lambda_0\}$.
\end{proposition}

By $\mathfrak{End}_{\langle\lambda,\sigma\rangle}(\mathscr{C}_{+}(a,b))$  we denote the subsemigroup of $\mathfrak{End}(\mathscr{C}_{+}(a,b))$, which is generated by the families  $\left\{\lambda_k\colon k{\in}\omega\right\}$ and $\left\{\sigma_{l,m}\colon l,m\in\omega,  m>0\right\}$ of endomorphisms of the monoid $\mathscr{C}_{+}(a,b)$. We summarise the results of this section in the following theorem.

\begin{theorem}\label{theorem-2.7}
\begin{enumerate}
  \item[(1)]\label{theorem-2.7-1} $\lambda_1$ is the identity element of the semigroup \linebreak $\mathfrak{End}(\mathscr{C}_{+}(a,b))$, and hence it is the identity element of the semigroups $\mathfrak{End}_{\langle\lambda\rangle}(\mathscr{C}_{+}(a,b))$ and $\mathfrak{End}_{\langle\lambda,\sigma\rangle}(\mathscr{C}_{+}(a,b))$;
  \item[(2)]\label{theorem-2.7-2} $\lambda_0$ is the zero of the semigroup $\mathfrak{End}(\mathscr{C}_{+}(a,b))$, and hence it is the zero of $\mathfrak{End}_{\langle\lambda\rangle}(\mathscr{C}_{+}(a,b))$ and $\mathfrak{End}_{\langle\lambda,\sigma\rangle}(\mathscr{C}_{+}(a,b))$;
  \item[(3)]\label{theorem-2.7-3} the set $I=\mathfrak{End}_{\langle\sigma\rangle}(\mathscr{C}_{+}(a,b))\cup\{\lambda_0\}$ is an ideal of the semigroup $\mathfrak{End}_{\langle\lambda,\sigma\rangle}(\mathscr{C}_{+}(a,b))$.
\end{enumerate}
\end{theorem}

\begin{proof}
Statements \eqref{theorem-2.7-1} and \eqref{theorem-2.7-2} are trivial.

\eqref{theorem-2.7-3} By Proposition~\ref{proposition-2.6} we have that $\sigma_{l,m}\lambda_k,\lambda_k\sigma_{l,m}\in \mathfrak{End}_{\langle\sigma\rangle}(\mathscr{C}_{+}(a,b))$ for any $\sigma_{l,m}\in \mathfrak{End}_{\langle\sigma\rangle}(\mathscr{C}_{+}(a,b))$ and  $\lambda_k\in \mathfrak{End}_{\langle\lambda\rangle}(\mathscr{C}_{+}(a,b))\setminus\{\lambda_0\}$. Since $\lambda_0$ is the zero of the semigroup $\mathfrak{End}(\mathscr{C}_{+}(a,b))$, the above arguments imply that $I$ is an ideal of the semigroup $\mathfrak{End}_{\langle\lambda,\sigma\rangle}(\mathscr{C}_{+}(a,b))$.
\end{proof}

\section{On monoid injective endomorphisms of $\mathscr{C}_{+}(a,b)$}\label{section-3}

\begin{example}\label{example-3.1}
For arbitrary positive integer $n$ and arbitrary ${s{=}0,\ldots,n{-}1}$ we define the mapping $\lambda_{n,s}\colon\mathscr{C}_{+}(a,b)\to \mathscr{C}_{+}(a,b)$ by the formula
\begin{equation*}
  (b^ia^j)\lambda_{n,s}=
  \left\{
    \begin{array}{ll}
      a^{nj},           & \hbox{if~} i=0;\\
      b^{ni-s}a^{nj-s}, & \hbox{if~} i\neq 0,
    \end{array}
  \right.
\end{equation*}
for all $i,j\in\omega$.
\end{example}

\begin{proposition}\label{proposition-3.2}
For any positive integer $n$ and any $s=0,\ldots,n-1$ the map $\lambda_{n,s}$  is an injective monoid endomorphism of the monoid $\mathscr{C}_{+}(a,b)$.
\end{proposition}

\begin{proof}
Fix any positive integers $i,j,k,l$ such that $i\leqslant j$ and $k\leqslant l$, and non-negative integers $m$ and $q$. Then we have that
\begin{align*}
  (b^ia^j\cdot b^ka^l)\lambda_{n,s}&=
\left\{
  \begin{array}{ll}
    (b^{i-j+k}a^{l})\lambda_{n,s}, & \hbox{if~} j<k; \\
    (b^ia^l)\lambda_{n,s},         & \hbox{if~} j=k; \\
    (b^ia^{j-k+l})\lambda_{n,s},   & \hbox{if~} j>k
  \end{array}
\right.=
\\
   &=
\left\{
  \begin{array}{ll}
    b^{n(i-j+k)-s}a^{nl-s},  & \hbox{if~} j<k; \\
    b^{ni-s}a^{nl-s},        & \hbox{if~} j=k; \\
    b^{ni-s}a^{n(j-k+l)-s},  & \hbox{if~} j>k,
  \end{array}
\right.
\end{align*}
\begin{align*}
  (b^ia^j)\lambda_{n,s}&\cdot (b^ka^l)\lambda_{n,s}=b^{ni-s}a^{nj-s}\cdot b^{nk-s}a^{nl-s}= \\
   &=
\left\{
  \begin{array}{ll}
    b^{(ni-s)-(nj-s)+(nk-s)}a^{nl-s},  & \hbox{if~} nj-s<nk-s; \\
    b^{ni-s}a^{nl-s},                  & \hbox{if~} nj-s=nk-s; \\
    b^{ni-s}a^{nj-s-(nk-s)+(nl-s)},    & \hbox{if~} nj-s>nk-s
  \end{array}
\right.=
\\
   &=
\left\{
  \begin{array}{ll}
    b^{n(i-j+k)-s}a^{nl-s},  & \hbox{if~} j<k; \\
    b^{ni-s}a^{nl-s},        & \hbox{if~} j=k; \\
    b^{ni-s}a^{n(j-k+l)-s},  & \hbox{if~} j>k,
  \end{array}
\right.
\end{align*}
\begin{align*}
  (b^ia^j\cdot a^m)\lambda_{n,s}&=(b^ia^{j+m})\lambda_{n,s}=\\
   &=b^{ni-s}a^{n(j+m)-s}= \\
   &=b^{ni-s}a^{nj-s}\cdot a^{mn}=\\
   &=(b^ia^j)\lambda_{n,s}\cdot (a^m)\lambda_{n,s},
\end{align*}
\begin{align*}
  (a^m\cdot b^ia^j)\lambda_{n,s}&=
\left\{
  \begin{array}{ll}
    (b^{i-m}a^j)\lambda_{n,s}, & \hbox{if~} m<i;\\
    (a^{m-i+j})\lambda_{n,s},  & \hbox{if~} m\geqslant i
  \end{array}
\right.=
\\
   &=
\left\{
  \begin{array}{ll}
    b^{n(i-m)-s}a^{nj-s}, & \hbox{if~} m<i;\\
    a^{n(m-i+j)},         & \hbox{if~} m\geqslant i
  \end{array}
\right.
\end{align*}
\begin{align*}
  (a^m)\lambda_{n,s}\cdot (b^ia^j)\lambda_{n,s}&= a^{mn}\cdot b^{ni-s}a^{nj-s}= \\
   &=
\left\{
  \begin{array}{ll}
    b^{ni-s-nm}a^{nj-s}, & \hbox{if~} mn<ni-s \\
     a^{mn-(ni-s)+nj-s}, & \hbox{if~} mn\geqslant ni-s
  \end{array}
\right.=
\\
\end{align*}
\begin{align*}
   &=
\left\{
  \begin{array}{ll}
    b^{n(i-m)-s}a^{nj-s}, & \hbox{if~} m<i-s/n;\\
    a^{n(m-i+j)},         & \hbox{if~} m\geqslant i-s/n
  \end{array}
\right. =
\\
   &=
\left\{
  \begin{array}{ll}
    b^{n(i-m)-s}a^{nj-s}, & \hbox{if~} m<i;\\
    a^{n(m-i+j)},         & \hbox{if~} m\geqslant i,
  \end{array}
\right.
\end{align*}
because $s=0,\ldots,n-1$, and
\begin{equation*}
  (a^m\cdot a^q)\lambda_{n,s}=(a^{m+q})\lambda_{n,s}=a^{(m+q)n}=a^{mn}\cdot a^{qn}=
  (a^m)\lambda_{n,s}\cdot (a^q)\lambda_{n,s}.
\end{equation*}
Hence the map $\lambda_{n,s}$  is a monoid endomorphism. The condition that $s=0,\ldots,n-1$ implies that $\lambda_{n,s}$ is an injective map.
\end{proof}

By $\mathfrak{End}_{\langle\lambda^\infty\rangle}(\mathscr{C}_{+}(a,b))$  we denote the subset of $\mathfrak{End}(\mathscr{C}_{+}(a,b))$, which consists of the elements of the family $\left\{\lambda_{n,s}\colon n\in\omega, \; s=0,\ldots,n-1\right\}$ of endomorphisms of the monoid $\mathscr{C}_{+}(a,b)$.

Let $S$ and $T$ be arbitrary semigroups. Let $\varphi\colon T\to\mathfrak{End}(S)$, $t\mapsto \varphi_t$ be a homomorphism from $T$ into the semigroup $\mathfrak{End}(S)$ of endomorphisms of $S$. The \emph{semidirect product} of $S$ and $T$ defined on the product $S\times T$ with the semigroup operation
\begin{equation*}
  (s_1,t_1)\cdot (s_2,t_2)=(s_1\cdot (s_2)\varphi_{t_1},t_1\cdot t_2),
\end{equation*}
and it is denoted by $S\rtimes_\varphi T$.

\begin{theorem}\label{theorem-3.3}
The set $\mathfrak{End}_{\langle\lambda^\infty\rangle}(\mathscr{C}_{+}(a,b))$ is a submonoid of $\mathfrak{End}(\mathscr{C}_{+}(a,b))$ and $\mathfrak{End}_{\langle\lambda^\infty\rangle}(\mathscr{C}_{+}(a,b))$ is
isomorphic to a submonoid of the semidirect product $(\mathbb{N},*)\rtimes_\varphi(\omega, +)$, where $(p)\varphi_n=np$.
\end{theorem}

\begin{proof}
Fix arbitrary positive integers $n_1$, $n_2$, $s_1=0,\dots,n_1-1$ and $s_2=0,\dots,n_2-1$. Then for any $b^ia^j\in \mathscr{C}_{+}(a,b)$ we have that
\begin{align*}
  (&(b^ia^j)\lambda_{n_1,s_1})\lambda_{n_2,s_2}=
  \left\{
    \begin{array}{ll}
      (a^{n_1j})\lambda_{n_2,s_2},                 & \hbox{if~} i=0;\\
      (b^{n_1i-s_1}a^{n_1j-s_1})\lambda_{n_2,s_2}, & \hbox{if~} i\neq 0,
    \end{array}
  \right.=\\
   &=
\left\{
    \begin{array}{ll}
      a^{n_1n_2j},                                & \hbox{if~} i=0;\\
      a^{n_2(n_1j-s_1)},                          & \hbox{if~} i\neq 0 \hbox{~and~} n_1i-s_1=0;\\
      b^{n_2(n_1i-s_1)-s_2}a^{n_2(n_1j-s_1)-s_2}, & \hbox{if~} i\neq 0 \hbox{~and~} n_1i-s_1\neq0,
    \end{array}
  \right.=\\
   &=
\left\{
    \begin{array}{ll}
      a^{n_1n_2j},                                      & \hbox{if~} i=0;\\
      b^{n_2n_1i-(s_1n_2+s_2)}a^{n_2n_1j-(s_1n_2+s_2)}, & \hbox{if~} i\neq 0
    \end{array}
  \right.=\\
   &=(b^ia^j)\lambda_{n_1n_2,s_1n_2+s_2}.
\end{align*}
Since $s_1<n_1$ and $s_2 < n_2$, we have that
\begin{equation*}
  s_1n_2+s_2\leqslant (n_1-1)n_2+s_2<(n_1-1)n_2+n_2=n_1n_2-n_2+n_2=n_1n_2,
\end{equation*}
and hence $s_1n_2+s_2<n_1n_2$. This implies that $\mathfrak{End}_{\langle\lambda^\infty\rangle}(\mathscr{C}_{+}(a,b))$ is a subsemigroup of $\mathfrak{End}(\mathscr{C}_{+}(a,b))$.
Since $\lambda_{n,0}=\lambda_{n}$ for any positive integer $n$, $\mathfrak{End}_{\langle\lambda\rangle}(\mathscr{C}_{+}(a,b))$ is a submonoid of $\mathfrak{End}_{\langle\lambda^\infty\rangle}(\mathscr{C}_{+}(a,b))$.

We define the map $\boldsymbol{\Phi}\colon \mathfrak{End}_{\langle\lambda^\infty\rangle}(\mathscr{C}_{+}(a,b))\to  (\mathbb{N},*)\rtimes_\varphi(\omega, +)$ by the formula $(\lambda_{n,s})\boldsymbol{\Phi}=(n,s)$. The above arguments imply that
\begin{align*}
  (\lambda_{n_1,s_1}\lambda_{n_2,s_2})\boldsymbol{\Phi}&=(\lambda_{n_1n_2,s_1n_2+s_2})\boldsymbol{\Phi}= \\
  &=(n_1n_2,s_1n_2+s_2)=\\
  &=(n_1n_2,(s_1)\varphi_{n_2}+s_2)=\\
  &=(n_1,s_1)(n_2,s_2)=\\
  &=(\lambda_{n_1,s_1})\boldsymbol{\Phi}(\lambda_{n_2,s_2})\boldsymbol{\Phi},
\end{align*}
and hence $\boldsymbol{\Phi}$ is a homomorphism.
\end{proof}

\begin{example}\label{example-3.4}
We define the map $\varsigma\colon \mathscr{C}_{+}(a,b)\to\mathscr{C}_{+}(a,b)$ by the formula
\begin{equation*}
  (b^ia^j)\varsigma=
\left\{
  \begin{array}{ll}
    1, & \hbox{if~} i=j=0;\\
    b^{i+1}a^{j+1}, & \hbox{otherwise},
  \end{array}
\right.
\end{equation*}
for any $i,j\in\omega$.
\end{example}

\begin{lemma}\label{lemma-3.5}
The map $\varsigma\colon \mathscr{C}_{+}(a,b)\to\mathscr{C}_{+}(a,b)$ is an injective monoid endomorphism.
\end{lemma}

\begin{proof}
A simple verification shows that $\varsigma$ is an injective map.
Obviously that it is sufficient to show that for any $b^ia^j\neq 1$ and $b^ka^l\neq 1$ the following equality
$(b^ia^j\cdot b^ka^l)\varsigma=(b^ia^j)\varsigma\cdot (b^ka^l)\varsigma$ holds. Indeed,
\begin{align*}
  (b^ia^j\cdot b^ka^l)\varsigma&=
\left\{
  \begin{array}{ll}
    (b^{i-j+k}a^l)\varsigma, & \hbox{if~} j<k;\\
    (b^ia^l)\varsigma,       & \hbox{if~} j=k;\\
    (b^ia^{j-k+l})\varsigma, & \hbox{if~} j>k
  \end{array}
\right.=
 \\
   &=
\left\{
  \begin{array}{ll}
    b^{i-j+k+1}a^{l+1}, & \hbox{if~} j<k;\\
    b^{i+1}a^{l+1},     & \hbox{if~} j=k;\\
    b^{i+1}a^{j-k+l+1}, & \hbox{if~} j>k
  \end{array}
\right.
\end{align*}
and
\begin{align*}
  (b^ia^j)\varsigma\cdot (b^ka^l)\varsigma &=b^{i+1}a^{j+1}\cdot b^{k+1}a^{l+1}= \\
\end{align*}
\begin{align*}
   &=
\left\{
  \begin{array}{ll}
    b^{i+1-(j+1)+k+1}a^{l+1}, & \hbox{if~} j+1<k+1;\\
    b^{i+1}a^{l+1},           & \hbox{if~} j+1=k+1; \\
    b^{i+1}a^{j+1-(k+1)+l+1}, & \hbox{if~} j+1>k+1
  \end{array}
\right.=
 \\
   &=
\left\{
  \begin{array}{ll}
    b^{i-j+k+1}a^{l+1}, & \hbox{if~} j<k;\\
    b^{i+1}a^{l+1},     & \hbox{if~} j=k;\\
    b^{i+1}a^{j-k+l+1}, & \hbox{if~} j>k,
  \end{array}
\right.
\end{align*}
and hence $\varsigma$ is an injective monoid endomorphism of $\mathscr{C}_{+}(a,b)$.
\end{proof}

By $\mathfrak{End}_{\langle\varsigma\rangle}(\mathscr{C}_{+}(a,b))$  we denote the subsemigroup of $\mathfrak{End}(\mathscr{C}_{+}(a,b))$, which is generated by endomorphism $\varsigma$ of the monoid $\mathscr{C}_{+}(a,b)$. Also by $\mathfrak{End}_{\langle\varsigma\rangle}^1(\mathscr{C}_{+}(a,b))$ we denote the semigroup $\mathfrak{End}_{\langle\varsigma\rangle}(\mathscr{C}_{+}(a,b))$ with the adjoined unit. Without loss of generality we may assume that
\begin{equation*}
\mathfrak{End}_{\langle\varsigma\rangle}^1(\mathscr{C}_{+}(a,b))=\mathfrak{End}_{\langle\varsigma\rangle}(\mathscr{C}_{+}(a,b))\cup\{\lambda_1\}.
\end{equation*}

\begin{proposition}\label{proposition-3.6}
The semigroup $\mathfrak{End}_{\langle\varsigma\rangle}(\mathscr{C}_{+}(a,b))$ is isomorphic to the additive semigroup of positive integers $(\mathbb{N},+)$, and hence $\mathfrak{End}_{\langle\varsigma\rangle}^1(\mathscr{C}_{+}(a,b))$ is isomorphic to the additive monoid of non-negative integers $(\omega,+)$.
\end{proposition}

\begin{proof}
For any $b^ia^j\in \mathscr{C}_{+}(a,b)$ and any positive integer $n$ by the definition of $\varsigma$ we have that
\begin{align*}
  (b^ia^j)\varsigma^n&=((b^ia^j)\varsigma)\varsigma^{n-1}= \\
   &=
   \left\{
  \begin{array}{ll}
    (1)\varsigma^{n-1},              & \hbox{if~} i=j=0;\\
    (b^{i+1}a^{j+1})\varsigma^{n-1}, & \hbox{otherwise},
  \end{array}
\right.=\\
   &=\ldots =\\
   &=
   \left\{
  \begin{array}{ll}
    (1)\varsigma,                  & \hbox{if~} i=j=0;\\
    (b^{i+n-1}a^{j+n-1})\varsigma, & \hbox{otherwise},
  \end{array}
\right.=\\
   &=
   \left\{
  \begin{array}{ll}
    1, & \hbox{if~} i=j=0;\\
    b^{i+n}a^{j+n}, & \hbox{otherwise}.
  \end{array}
\right.
\end{align*}
The definition of the bicyclic monoid $\mathscr{C}(a,b)$ implies that $b^{k_1}a^{l_1}=b^{k_2}a^{l_2}$ in $\mathscr{C}_{+}(a,b)$ if an only if $k_1=k_2$ and $l_1=l_2$. This and above equalities imply that the endomorphism $\varsigma$ generates the infinite cyclic subsemigroup in $\mathfrak{End}(\mathscr{C}_{+}(a,b))$, and hence $\mathfrak{End}_{\langle\varsigma\rangle}(\mathscr{C}_{+}(a,b))$ is isomorphic to the additive semigroup of positive integers $(\mathbb{N},+)$. The last statement of the proposition is obvious.
\end{proof}

\begin{lemma}\label{lemma-3.7}
Let $\varepsilon\colon \mathscr{C}_{+}(a,b)\to\mathscr{C}_{+}(a,b)$  be an injective monoid endomorphism such that $(a)\varepsilon=a^n$ for some positive integer $n$. Then there exists $s\in\{0,\ldots,n-1\}$ such that $\varepsilon=\lambda_{n,s}$.
\end{lemma}

\begin{proof}
We observe that if $\varepsilon$ is an injective endomorphism of $\mathscr{C}_{+}(a,b)$ then for any idempotents $b^ia^i,b^ja^j\in \mathscr{C}_{+}(a,b)$ the inequality $b^ia^i\preccurlyeq b^ja^j$ implies that $(b^ia^i)\varepsilon\preccurlyeq (b^ja^j)\varepsilon$, because the equality $b^ia^i\cdot b^ja^j=b^ia^i$ implies that
\begin{equation*}
(b^ia^i)\varepsilon\cdot(b^ja^j)\varepsilon=(b^ia^i \cdot b^ja^j)\varepsilon=(b^ia^i)\varepsilon.
\end{equation*}
Since $\varepsilon$ is an injective monoid  endomorphism of $\mathscr{C}_{+}(a,b)$, we conclude that $(b^ia^i)\varepsilon\neq(b^ja^j)\varepsilon$ and $(b^0a^0)\varepsilon=(1)\varepsilon=1=b^0a^0$. Hence  there exists a strictly increasing sequence $\{s_i\}_{i\in\omega}$ in $\omega$ such that  $(b^ia^i)\varepsilon=b^{s_i}a^{s_i}$ for any $i\in\omega$ and $s_0=0$. Then for any positive integer $i$ we have that
\begin{align*}
  (b^ia^{i+1})\varepsilon&=(b^ia^i\cdot a)\varepsilon=\\
	&=(b^ia^i)\varepsilon\cdot(a)\varepsilon=\\
    &=b^{s_i}a^{s_i}\cdot a^n=\\
	&=b^{s_i}a^{s_i+n}
\end{align*}
and
\begin{align*}
(b^ia^{i+1})\varepsilon&=(a\cdot b^{i+1}a^{i+1})\varepsilon=\\
	&=(a)\varepsilon\cdot (b^{i+1}a^{i+1})\varepsilon=\\
	&=a^n\cdot b^{s_{i+1}}a^{s_{i+1}}=\\
	&=
	\left\{
	\begin{array}{ll}
		a^{n}, & \hbox{if~} n\geqslant s_{i+1};\\
		b^{s_{i+1}-n}a^{s_{i+1}}, & \hbox{if~} n<s_{i+1}.
	\end{array}
	\right.
\end{align*}
The injectivity of $\varepsilon$ implies that $n<s_{i+1}$.
Since $s_0=0$ and $s_i< s_{i+1}$ for any $i\in\omega$, the above equalities imply that $s_{i+1}=s_i+n$ for any positive integer $i$, and hence $s_{i+1}-s_i=n$ for $i\in\mathbb{N}$. By induction we have that $s_{i+1}=s_1+in$ for all $i\in\mathbb{N}$. This implies that
\begin{equation*}
(b^ia^i)\varepsilon=b^{(i-1)n+s_1}a^{(i-1)n+s_1}
\end{equation*}
for any positive integer $i$. Then we get that
\begin{align*}
  (b^ia^{i+l})\varepsilon&=(b^ia^i\cdot a^l)\varepsilon=\\
   &=(b^ia^i)\varepsilon\cdot (a^l)\varepsilon= \\
   &=b^{(i-1)n+s_1}a^{(i-1)n+s_1}\cdot ((a)\varepsilon)^l= \\
   &=b^{(i-1)n+s_1}a^{(i-1)n+s_1}\cdot (a^n)^l= \\
   &=b^{(i-1)n+s_1}a^{(i-1)n+s_1}\cdot a^{nl}= \\
   &=b^{(i-1)n+s_1}a^{(i+l-1)n+s_1},
\end{align*}
for any $i,l\in\omega$. Also, since
\begin{align*}
  a^n &=(a)\varepsilon= \\
   &=(a\cdot ba)\varepsilon= \\
   &=(a)\varepsilon\cdot (ba)\varepsilon= \\
   &=a^n\cdot b^{s_1}a^{s_1},
\end{align*}
properties of semigroup operation of $\mathscr{C}_{+}(a,b)$ and the natural partial order on the set of idempotents of the monoid $\mathscr{C}_{+}(a,b)$ imply that $s_1\leqslant n$, and by injectivity of $\varepsilon$ we have that $s_1>0$.   Put $s=n-s_1$. Then we obtain that $s\in\{0,\ldots,n-1\}$ and
\begin{align*}
  (b^ia^{i+l})\varepsilon&=b^{(i-1)n+s_1}a^{(i+l-1)n+s_1}= \\
  &=b^{in-n+s_1}a^{(i+l)n-n+s_1}=\\
  &=b^{in-s}a^{(i+l)n-s}.
\end{align*}
This and Proposition \ref{proposition-3.2} imply that $\varepsilon=\lambda_{n,s}$.
\end{proof}

\begin{proposition}\label{proposition-3.8}
Let $S$, $T$, and $U$ be  semigroups, let $\mathfrak{h}\colon S\to T$ be a homomorphism, and let $\mathfrak{g}\colon U\to T$ be an injective homomorphism. If $(S)\mathfrak{h}\subseteq (U)\mathfrak{g}$, then the mapping $\mathfrak{f}\colon S\to U$ which is defined by the formula $(s)\mathfrak{f}=((s)\mathfrak{h})\mathfrak{g}^{-1}$ is a homomorphism. Moreover, if the homomorphism $\mathfrak{h}\colon S\to T$ is injective or a monoid homomorphism, then so is $\mathfrak{f}$, too.
\end{proposition}

\begin{proof}
Since $\mathfrak{g}\colon U\to T$ is an injective homomorphism, the map $\mathfrak{f}\colon S\to U$ is well defined. Also, for arbitrary $s_1,s_2\in S$ we have that
\begin{align*}
  (s_1\cdot s_2)\mathfrak{f}& =((s_1\cdot s_2)\mathfrak{h})\mathfrak{g}^{-1}= \\
   &=((s_1)\mathfrak{h}\cdot(s_2)\mathfrak{h})\mathfrak{g}^{-1}= \\
   &=((s_1)\mathfrak{h})\mathfrak{g}^{-1}\cdot((s_2)\mathfrak{h})\mathfrak{g}^{-1}=\\
   &=(s_1)\mathfrak{f}\cdot (s_2)\mathfrak{f},
\end{align*}
because $\mathfrak{g}\colon U\to T$ is an injective homomorphism. Hence $\mathfrak{f}\colon S\to U$ is s homomorphism. The second statement of the proposition is obvious.
\end{proof}

\begin{lemma}\label{lemma-3.9}
Let $\varepsilon\colon \mathscr{C}_{+}(a,b)\to\mathscr{C}_{+}(a,b)$  be an injective monoid endomorphism such that $(a)\varepsilon=b^na^{n+p}$ for some positive integers $n$ and $p$. Then there exists $s\in\{0,\ldots,n-1\}$ such that $\varepsilon=\lambda_{p,s}\varsigma^n$.
\end{lemma}

\begin{proof}
Since $\varepsilon$ is an injective monoid  endomorphism of $\mathscr{C}_{+}(a,b)$, arguments presented in the proof of Lemma~\ref{lemma-3.7} imply that there exists a strictly increasing sequence $\{s_i\}_{i\in\omega}$ in $\omega$ such that  $(b^ia^i)\varepsilon=b^{s_i}a^{s_i}$ for any $i\in\omega$ and $s_0=0$. Then for any positive integers $i$ and $j$ we have that
\begin{align*}
  (b^{i}a^{i+j})\varepsilon& =(b^{i}a^{i}\cdot a^j)\varepsilon= \\
   &=(b^{i}a^{i})\varepsilon\cdot (a^j)\varepsilon= \\
   &=b^{s_i}a^{s_i}\cdot (b^na^{n+p})^j= \\
   &=b^{s_i}a^{s_i}\cdot b^na^{n+pj}= \\
   &=
	\left\{
	\begin{array}{ll}
		b^na^{n+pj}, & \hbox{if~} n\geqslant s_{i};\\
		b^{s_{i}}a^{s_{i}+pj}, & \hbox{if~} n<s_{i}.
	\end{array}
	\right.
\end{align*}
This implies that $(\mathscr{C}_{+}(a,b))\varepsilon\subseteq (\mathscr{C}_{+}(a,b))\varsigma^n$. By Proposition \ref{proposition-3.8} the mapping $\mathfrak{f}\colon \mathscr{C}_{+}(a,b)\to \mathscr{C}_{+}(a,b)$ which is defined by the formula $(b^ia^j)\mathfrak{f}=((b^ia^j)\varepsilon)(\varsigma^n)^{-1}$ is an endomorphism of the monoid $\mathscr{C}_{+}(a,b)$. Simple verifications show that $(a)\mathfrak{f}=a^p$. By Lemma~\ref{lemma-3.7} we have that $\mathfrak{f}=\lambda_{p,s}$ for some $s\in\{0,\ldots,n-1\}$, i.e., $\lambda_{p,s}=\varepsilon(\varsigma^n)^{-1}$. Since $\varsigma$ is an injective monoid endomorphism of $\mathscr{C}_{+}(a,b)$, we conclude that that $\varepsilon=\lambda_{p,s}\varsigma^n$.
\end{proof}

Theorem~\ref{theorem-3.10} describes all injective endomorphisms of the semigroup $\mathscr{C}_{+}(a,b)$ and it follows from Lemmas~\ref{lemma-3.7} and \ref{lemma-3.9}.

\begin{theorem}\label{theorem-3.10}
Let $\varepsilon\colon \mathscr{C}_{+}(a,b)\to\mathscr{C}_{+}(a,b)$  be an injective monoid endomorphism. Then only  one of the following statements holds:
\begin{enumerate}
  \item[(1)] there exist a positive integer $p$ and $s\in\{0,\ldots,p-1\}$ such that $\varepsilon=\lambda_{p,s}$;
  \item[(2)] there exist  positive integers $n$, $p$, and $s\in\{0,\ldots,p-1\}$ such that $\varepsilon=\lambda_{p,s}\varsigma^n$.
\end{enumerate}
\end{theorem}

It is natural to ask the following: what is a semigroup operation on the subsemigroup $\mathfrak{End}_{\langle\overline{\lambda\varsigma}\rangle}(\mathscr{C}_{+}(a,b))$ of $\mathfrak{End}(\mathscr{C}_{+}(a,b))$ of endomorphisms of $\mathscr{C}_{+}(a,b)$ which is generated by endomorphism of the form $\lambda_{p,s}\varsigma^{n}$, where $p,n\in\mathbb{N}$ and $s\in\{0,\ldots,n-1\}$.
Theorem \ref{theorem-3.11} describes the structure of the semigroup operation on the semigroup $\mathfrak{End}_{\langle\overline{\lambda\varsigma}\rangle}(\mathscr{C}_{+}(a,b))$ of injective monoid endomorphisms of $\mathscr{C}_{+}(a,b)$.

\begin{theorem}\label{theorem-3.11}
$\mathfrak{End}_{\langle\overline{\lambda\varsigma}\rangle}(\mathscr{C}_{+}(a,b))$ is a subsemigroup of $\mathfrak{End}(\mathscr{C}_{+}(a,b))$  and it is isomorphic to the  subsemigroup of the Cartesian power $\mathbb{N}^3$ with the following semigroup operation
\begin{equation*}
  (p_1,s_1,n_1)\cdot(p_2,s_2,n_2)=(p_2p_1,p_2s_1,p_2n_1-s_2+n_2).
\end{equation*}
\end{theorem}

\begin{proof}
Fix arbitrary $p_1,n_1,p_2,n_2\in\mathbb{N}$, $s_1\in\{0,\ldots,n_1-1\}$ and $s_2\in\{0,\ldots,n_2-1\}$. Then for any  $(b^ia^j\in\mathscr{C}_{+}(a,b)$ we have that
\begin{align*}
  (&(((b^ia^j)\lambda_{p_1,s_1})\varsigma^{n_1})\lambda_{p_2,s_2})\varsigma^{n_2}=\\
  &=
  \left\{
    \begin{array}{ll}
      (((a^{p_1j})\varsigma^{n_1})\lambda_{p_2,s_2})\varsigma^{n_2},                 & \hbox{if~} i=0;\\
      (((b^{p_1i-s_1}a^{p_1j-s_1})\varsigma^{n_1})\lambda_{p_2,s_2})\varsigma^{n_2}, & \hbox{if~} i\neq 0
    \end{array}
  \right.=\\
   &=
   \left\{
    \begin{array}{ll}
      ((1)\lambda_{p_2,s_2})\varsigma^{n_2},                                & \hbox{if~} i=0 \hbox{~and~} j=0;\\
      ((b^{n_1}a^{p_1j+n_1})\lambda_{p_2,s_2})\varsigma^{n_2},              & \hbox{if~} i=0 \hbox{~and~} j\neq0;\\
      ((b^{p_1i-s_1+n_1}a^{p_1j-s_1+n_1})\lambda_{p_2,s_2})\varsigma^{n_2}, & \hbox{if~} i\neq 0
    \end{array}
  \right.=\\
   &=
  \left\{
    \begin{array}{ll}
      (1)\varsigma^{n_2},                                                  & \hbox{if~} i=0 \hbox{~and~} j=0;\\
      (b^{p_2n_1-s_2}a^{p_2(p_1j+n_1)-s_2})\varsigma^{n_2},                & \hbox{if~} i=0 \hbox{~and~} j\neq0;\\
      (b^{p_2(p_1i-s_1+n_1)-s_2}a^{p_2(p_1j-s_1+n_1)-s_2})\varsigma^{n_2}, & \hbox{if~} i\neq 0
    \end{array}
  \right.=\\
   &=
   \left\{
    \begin{array}{ll}
      1,                                                           & \hbox{if~} i=0 \hbox{~and~} j=0;\\
      b^{p_2n_1-s_2+n_2}a^{p_2(p_1j+n_1)-s_2+n_2},                 & \hbox{if~} i=0 \hbox{~and~} j\neq0;\\
      b^{p_2(p_1i-s_1+n_1)-s_2+n_2}a^{p_2(p_1j-s_1+n_1)-s_2+n_2},  & \hbox{if~} i\neq 0
    \end{array}
  \right.=\\
   &=\!
  \left\{\!\!
    \begin{array}{ll}
      1,                                                           & \hbox{if~} i=0 \hbox{~and~} j=0;\\
      b^{p_2n_1-s_2+n_2}a^{p_2p_1j+p_2n_1-s_2+n_2},                 & \hbox{if~} i=0 \hbox{~and~} j\neq0;\\
      b^{p_2p_1i-p_2s_1+p_2n_1-s_2+n_2}a^{p_2p_1j-p_2s_1+p_2n_1-s_2+n_2},  & \hbox{if~} i\neq 0
    \end{array}
  \right.\!\!\!\! \\
   &=(b^ia^j)\lambda_{p_2p_1,p_2s_1}\varsigma^{p_2n_1-s_2+n_2}
\end{align*}
The above equalities imply the statement of the theorem.
\end{proof}

\begin{remark}
\begin{enumerate}
  \item The endomorphisms $\lambda_k$, $k\in\omega$, $\lambda_{p,s}$, $p,n\in\mathbb{N}$, $s\in\{0,\ldots,n-1\}$ and $\varsigma$ of the monoid $\mathscr{C}_{-}(a,b)$ are introduced by the same formulae.
  \item Repeating the proofs of corresponding statements in Sections~\ref{section-2} and~\ref{section-3} we obtain that the same statements about the corresponding endomorphisms of the monoid $\mathscr{C}_{-}(a,b)$. Moreover, the corresponding semigroups of endomorphisms are pairwise isomorphic.
\end{enumerate}
\end{remark}

\section*{Acknowledgements}

The authors acknowledge Alex Ravsky and the referee  for they comments and suggestions.


\begin{thebibliography}{11}

\bibitem{Clifford-Preston-1961}
A. H.~Clifford and  G. B.~Preston,
\emph{The algebraic theory of semigroups},
Vol. I., Amer. Math. Soc. Surveys 7, Pro\-vi\-den\-ce, R.I., 1961.

\bibitem{Clifford-Preston-1967}
A. H.~Clifford and  G. B.~Preston,
\emph{The algebraic theory of semigroups},
Vol. II., Amer. Math. Soc. Surveys 7, Provi\-den\-ce, R.I., 1967.

\bibitem{Descalco-Ruskuc-2005}
L. Descal\c{c}o and N. Ru\v{s}kuc,
\emph{Subsemigroups of the bicyclic monoid},
Int. J. Algebra Comput. \textbf{15} (2005), no. 1, 37--57.
DOI: 10.1142/S0218196705002098

\bibitem{Descalco-Ruskuc-2008}
L. Descal\c{c}o and N. Ru\v{s}kuc,
\emph{Properties of the subsemigroups of the bicyclic monoid},
Czech. Math. J. \textbf{58} (2008), no.~2, 311--330.
DOI: 10.1007/s10587-008-0018-7


\bibitem{Gutik=2023}
O. Gutik,
\emph{On non-topologizable semigroups},
Preprint (arXiv:2405.16992).


\bibitem{Gutik-Prokhorenkova-Sekh=2021}
O. Gutik, O. Prokhorenkova, and D. Sekh,
\emph{On endomorphisms of the bicyclic semigroup and the extended bicyclic semigroup},
Visn. L'viv. Univ., Ser. Mekh.-Mat. \textbf{92} (2021), 5--16  (in Ukrainian).
DOI: 10.30970/vmm.2021.92.005-016



\bibitem{Lawson=1998}
M.~Lawson,
\emph{Inverse semigroups. The theory of partial symmetries},
World Scientific, Singapore, 1998.

\bibitem{Makanjuola-Umar=1997}
S. O. Makanjuola and A. Umar,
\emph{On a certain sub semigroup of the bicyclic semigroup},
Commun. Algebra  \textbf{25} (1997), no. 2, 509-519,
DOI: 10.1080/00927879708825870

\bibitem{Wagner-1952}
V.~V. Wagner,
\textit{Generalized groups},
Dokl. Akad. Nauk SSSR \textbf{84} (1952), 1119--1122 (in Russian).

\end{thebibliography}
\end{document}